\documentclass[reqno,b5paper]{amsart}
\usepackage{amsmath}
\usepackage{amssymb}
\usepackage{amsthm}
\usepackage{enumerate}
\usepackage[mathscr]{eucal}
\theoremstyle{plain}
\newtheorem{thm}{Theorem}[section]
\newtheorem{corr}[thm]{Corollary}

\newtheorem{Example}{Example}[section]

\theoremstyle{definition}
\newtheorem{defn}{Definition}[section]

\begin{document}

\setcounter {page}{1}
%---------------Title,Author,Abstract-----------------------------------------------
\title{New types of convergence on time scales}

\author[A. Ghosh and M. Maity]{ Argha Ghosh* and Manojit Maity**\ }
\newcommand{\acr}{\newline\indent}
\maketitle
\address{{*\,}  The Department of Mathematics, The University of Burdwan, Golapbag, Pin-713104, India. Email: buagbu@yahoo.co.in\acr
            {**\,} Assistant Teacher in Mathematics, Boral High School, Kolkata-700154, India. Email: mepsilon@gmail.com\\}

\maketitle
\begin{abstract}
This paper is discussing about the notion of some new types of convergences namely $I$-convergence and $I^*$-convergence of a $\Delta$-measurable function $f$ on time scales $T$ by considering ideal on time scales $T$. This idea is further extended to the notion of statistical convergence of a $\Delta $-measurable function $f$ on time scales $T$.
\end{abstract}
\author{}
\maketitle
{ Key words and phrases :} $\Delta$- measurable function, time scale, $I$-convergence on time scale, $I^*$-convergence on time scale. \\

\textbf {AMS subject classification (2010) :} 26E70, 40A35.  \\

%-------------------------Section 1- Background and introduction-----------------------
\section{\textbf{Introduction}}

 Let $\mathbb R$ be the set of real numbers. By a time scale $ T$, we mean a nonempty closed subset of $\mathbb R$. The forward jump operator $\sigma :T \rightarrow T$  can be defined by $\sigma (t)=inf \left\{s\in T:s>t\right\}$, $t\in T$ and the graininess function $\mu :T\rightarrow [0,\infty)$
 can be defined by $\mu(t)=\sigma(t)-t$. A half open interval on an arbitrary time scale T is given by $[a,b)_T=\left\{t:a\leq t<b\right\}_T$. Open intervals or closed intervals can be defined similarly (see [10]).
 
The notion of the time scales calculus were first considered by Hilger $[11]$ are to fulfill demands that one unify the discrete and continuous cases and to obtain some new notions. This method of calculus is effective in modeling some practical life problems for example one needs both discrete and continuous time variables to modeling prey and predator populations. Also a large numbers of very important functions on time scales have been applied to solve various dynamic equations, the expression of Green's functions of some boundary value problem $[3]$ or oscillation properties of first and second order nonlinear equations (see [7]).

 The idea of statistical convergence was first introduced by Fast [8] and Steinhus [16] independently, as an extension of the notion of ordinary convergence. To fill the gap between time scale and summability theory  C. Turan and O. Duman introduced the concept of statistical convergence on time scales [17] in 2011. So in view of recent applications of time scales in real life problems, it seems very natural to extend the interesting concept of convergence on time scales by using ideals which we mainly do here.

 By using ideals of the set of all natural numbers $\mathbb N$, Kostryko et. al [13] generalized the notion of statistical convergence of a sequence of real number and named as $I$ and $I^*$ convergence. A lot of seminal work has been done since then as can be found in [4, 5, 6, 12, 14]. However, so far, there is no usage of ideals of any set other than the set $\mathbb N$ and $\mathbb{N}\times\mathbb N$ in literature.

 This paper propose a new notion of convergence of a $\Delta$-measurable function by using ideals of time scales and to generate a new research area, the applications of ideals  
of a set other than the set $\mathbb N$ and $\mathbb{N}\times\mathbb N$ will motivate some researcher in future. As an immediate application of this notion of convergence as follows if the tastes of agents on a time scales $T$ are represented by a preference relation on $T$, that is a subset $R$ of $T\times T$ where $(x,y)\in R$ means the agent prefer alternative $x$ to $y$. Similarities of agents can be described by a convergence of a $\Delta$-measurable functions.

%------------------------------Section-2 - Related definitions-------------------------
\section{\textbf{Preliminaries  }}
%------------------------------------introduction to section2-----------------------
First we recall some basic concepts related to time scales and summability theory. We should note that throughout the paper, we consider that $T$ is a time scales satisfying inf $T=t_0>0$ and sup $T=\infty$. 

 Let $\mathcal F_1$ denote the family of all left closed and right open intervals of $T$ of the form $[a,b )_T$. Let $ m_1: \mathcal F_1 \rightarrow[0,\infty)$
be the set function on $\mathcal F_1$ such that $ m_1([a, b)_T ) =b-a$. Then, it is known that $m_1$ is a countably additive measure
on $\mathcal F_1$. Now, the Caratheodory extension of the set function $m_1$ associated with family $\mathcal F_1$ is said to be the Lebesgue $\Delta$-measure
on $T$ and is denoted by $\mu_\Delta$ (see [10, 15]). \\ In this case, it is known that if $a \in T \setminus\left\{max T\right\},$ then the single point set $\left\{a\right\}$ is $\Delta$-measurable and $\mu_\Delta(a)$ = $\sigma(a)-a$. If  $a, b \in T$ and $a\leq b$, then $\mu_\Delta((a, b)_T)$ =$b-\sigma(a)$. If $a, b \in T \setminus\left\{max T\right\}$, $a\leq b$; $\mu_\Delta((a,b ]_T ) = \sigma(b)-\sigma(a)$ and $\mu_\Delta([a, b]_T) = \sigma(b) - a$ (see [10]).
%---------------------------Definition 2.1------------------------------------------
\begin{defn}$[17]$
Let $\Omega $ be a $\Delta$-measurable subset of $T$. Then, for $t\in T$, we define the set $\Omega(t)$ by
\begin{center}
 $\Omega(t)=\left\{s\in[t_0,t]_T: s\in\Omega\right\}$.
\end{center}
In this case, we define the density of $\Omega$ on $T$, denoted by $\delta_T(\Omega)$, as follows: 
\begin{center}
$\delta_T(\Omega):=\underset{t\rightarrow\infty}{\lim}\frac{\mu_\Delta(\Omega(t))}{\mu_\Delta([t_0,t]_T)}$
\end{center}
provided that the above limit exists.
\end{defn}
% --------------------------Definition 2.2-------------------------------------------
\begin{defn}$[17]$
Let $f:T\rightarrow\mathbb R$ be a $\Delta $-measurable function. We say that $f$ is statistically convergent on $T$ to a number $L$ if, for $\epsilon>0$,
\begin{center}
 $\delta_T(\left\{t\in T:\left|f(t)-L\right|\geq \epsilon\right\})=0$
\end{center}
holds.
\end{defn}
%-----------------------------Definition 2.3------------------------------------------
\begin{defn}$[13]$
 Let $X\neq\phi $. A class $ I $ of subsets of X is said to be
an ideal in X provided, $I$ satisfies the conditions:
\\(i)$\phi \in I$,
\\(ii)$ A,B \in I \Rightarrow A \cup B\in I,$
\\(iii)$ A \in I, B\subset A \Rightarrow B\in I$.
\end{defn}
An ideal $I$ in a non-empty set $X$ is called non-trivial if X $\notin I.$
\begin{defn}$[13]$
%...............................def 2.4......................................
 Let $X\neq\phi  $. A non-empty class $\mathbb F $ of subsets of X is
said to be a filter in X provided that:
\\(i)$\phi\notin \mathbb F $,
\\(ii) $A,B\in\mathbb F \Rightarrow A \cap B\in\mathbb F,$
\\(iii)$ A \in\mathbb F, B\supset A \Rightarrow B\in\mathbb F$.
\end{defn}
 %..............................defnition 2.5....................................
\begin{defn}
 Let $I$ be a non-trivial ideal in a non-empty set $ X$.
 Then the class $\mathbb F(I)$$ = \left\{M\subset X : \exists A \in I ~such~~ that ~M = X\setminus A\right\}$
is a filter on X. This filter $\mathbb F(I) $ is called the filter associated with $I$.
\end{defn}
%............................definition 2.6.................................
\begin{defn}$[13]$
A sequence $x=\left\{x_{k}\right\}_{k\in \mathbb{N}}$ of real numbers is said to converge to $\eta\in \mathbb{R}$ with respect to the ideal $I$, if for every $\epsilon > 0$ the set $A(\epsilon)=\left\{k:\left|x_k-\eta\right|\geq \epsilon\right\}\in I.$
\end{defn}
%-----------------------------------definition 2.7--------------------------------------
\begin{defn}$[13]$
A sequence $x=\left\{x_k\right\}_{k\in \mathbb{N}}$ of real numbers is said to $I^*$-converge to $\eta\in \mathbb{R}$, if there exists a set $M\in\mathbb F(I)$ where the sequence $x$ ordinary converges to $\eta$.
\end{defn}

%---------------------------------Section 3(i-convergence on time scale )-------------
\section{\textbf{Main Results }}
It is the main section of our paper, we focus on constructing a notion of $I$-convergence and  $I^*$-convergence on time scales. To do this we consider the measurable space $(T,M(m_1^*))$ equipped with Lebesgue $\Delta$-measure $\mu_\Delta$. We first give  some examples of ideals of time scales.
\begin{Example}
  $I=\left\{A\in M(m_1^*) :\mu_\Delta (A)=0\right\}$ is an ideal on a time scale.
\end{Example}
\begin{Example}
$I_\Delta $=$\left\{A\in M(m_1^*):\delta_T(A)=0\right\}$ is an ideal on a time scale.
\end{Example}
\begin{defn}
 An ideal $I$ of a time scale $T$ is said to be $B$-admissible if it contains all bounded subsets of $T$. 
\end{defn}
 \begin{defn}
 Let $f:T\rightarrow \mathbb R$ be a $\Delta$-measurable function. We say that $f$ is $I$-convergent on $T$ to a number
$L$ if for every $\epsilon>0$, 
\begin{center}
$A(\epsilon)=\left\{t\in T:\left|f(t)-L\right|\geq\epsilon\right\}\in I.$
\end{center}
\end{defn}
Consider the ideal $I_\Delta $=$\left\{A\in M(m_1^*):\delta_T(A)=0\right\}.$

Note that when $I=I_\Delta$, $I$-convergence reduced to statistical convergence on time scale $T$.
%.................................theo3.1..............................................
\begin{thm}
Let $f,g:T\rightarrow \mathbb R$ be a $\Delta$-measurable functions and $\alpha \in \mathbb R$. Then 
(1)if $I-lim f(t)=L$ and $I-lim f(t)=M$ then $L=M$.
\\(2)if $I-limf(t)=L$ then $I-lim~\alpha f=\alpha L$.
\\(3)if $I-limf(t)=L$ and $I-lim g(t)=M$ then $I-lim(f+g)=L+M$ and $I-lim(fg)=L M$.
\end{thm}
\begin{proof}
(1) Let $L\neq M$ then there exists an $\epsilon>0$ such that $(L-\epsilon ,L+\epsilon)\cap(M-\epsilon,M+\epsilon)=\phi.$ Since 
\begin{center}
$\left\{t:\left|f(t)-L\right|\geq\epsilon\right\}\in I$
\end{center}
 and
 \begin{center}
 $\left\{t:\left|f(t)-M\right|\geq\epsilon\right\}\in I$
 \end{center}
 we have 
 \begin{center}
 $\left\{t:f(t)\in\left[\left(L-\epsilon,L+\epsilon\right)\cap\left(M-\epsilon,M+\epsilon\right)\right]^c\right\}\subset\left\{t:f(t)\in(L-\epsilon,L+\epsilon)^c\right\}\cup\left\{t:f(t)\in(M-\epsilon,M+\epsilon)^c\right\}\in I$
 \end{center}
	where c stands for complement of the set. Since $I$ is non trivial there exists $t_0\in T$ such that 
 \begin{center}
 $t_0\notin\left\{t:f(t)\in\left[(L-\epsilon,L+\epsilon)\cap(M-\epsilon,M+\epsilon)\right]^c\right\}.$
 \end{center}
But then $f(t_0)\in\left\{(L-\epsilon,L+\epsilon)\cap(M-\epsilon,M+\epsilon)\right\}$ which is a contradiction and hence the theorem.
\end{proof}
The proof of (2) and (3) are trivial so omitted.
\begin{defn}
 Let $f:T\rightarrow\mathbb R$ be a $\Delta $-measurable function, we say that $f$ is $I^*$-convergent on $T$ to a number $L$ if there exists a set 
 $\Omega\in \mathbb F(I) $ such that $\underset{t\rightarrow\infty}{\lim}f(t)=L$, $t\in \Omega$ where 
$\mathbb F(I)=\left\{A\subseteq T:T\setminus A\in I\right\}$.
\end{defn}
%......................................theo3.2...................................................................
\begin{thm}
Let $f$ be a $\Delta $-measurable function then $I^*$-convergence of $f$ implies $I$-convergence of $f$ when $I$ is an $B$-admissible ideal. 
\end{thm}
\begin{proof}  
Let $\epsilon>0$ be given then there exists a number $t^*\in T$ such that for every $t\geq t^*$ with $t\in \Omega\in\mathbb F(I)$ one can obtain that
 $\left|f(t)-L\right|<\epsilon$. Hence if we put $A(\epsilon)=\left\{ t\in T: \left|f(t)-L\right|\geq\epsilon\right\}$ and $B=\Omega\cap [t^*,\infty)$ then $A(\epsilon)\subset T\setminus B$. Now since $I$ is an $B$-admissible ideal, so $[t_0,t^*)\in I$ and hence $[t^*,\infty)\in\mathbb F(I)$. Therefore $B=\Omega\cap [t^*,\infty)\in F(I)$ so $T\setminus B\in I$. Since $A\subset T\setminus B$ we have $A(\epsilon)\in I$. Therefore $I-limf(t)=L$. Thus $f$ is $I$-convergent.
\end{proof}
\begin{thm}
 A continuous function $g:\mathbb R\rightarrow \mathbb R$ preserves $I$-convergence .
\end{thm}
\begin{proof}
 Let $I-lim f(t)=L$ and $\epsilon>0$ be given. Consider the open interval $(g(L)-\epsilon,g(L)+\epsilon)$. Since $g$ is continuous at $L$ then 
 \begin{center}
 $g((L-\epsilon,L+\epsilon))\subseteq(g(L)-\epsilon,g(L)+\epsilon)$.
 \end{center}
 Clearly 
 \begin{center}
 $\left\{t:g(f(t))\notin(g(L)-\epsilon,g(L)+\epsilon)\right\} \subseteq \left\{t;f(t)\notin (L-\epsilon,L+\epsilon)\right\} $
 \end{center}
	and since
	\begin{center}
	 $\left\{t:f(t)\notin(L-\epsilon,L+\epsilon)\right\}\in I$ 
	\end{center}
	we have
	\begin{center}
	 $\left\{t:g(f(t))\notin(g(L)-\epsilon,g(L)+\epsilon\right\}\in I$.
	\end{center}
	Which shows that $I-lim g(f(t))=L$ i.e $g$ preserves $I$-convergence.
\end{proof}
%....................................theo3.3,,,,,,,,,,,,,,,,,,,,,,,
\begin{thm}
Let $f, g, h$ are all $\Delta$-measurable function on $T$. If $f(t)\leq g(t)\leq h(t)$ for all $t \in T$ and if 
\begin{center}
$I^*-lim f(t)=L=I^*-lim h(t) $
\end{center}
 then $I^*-lim g(t)=L$.
\end{thm}
\begin{proof}
 Let $\epsilon > 0$ be given. Then there exist $\Omega_1,\Omega_2 \in \mathbb F(I)$ and $t_1^*,t_2^*\in T$ such that  for all $t\geq t_1^* $ and $t\in \Omega _1$
 \begin{center}
	$L-\epsilon<f(t)<L+\epsilon$
 \end{center}
	and for all $t\geq t_2^*$ and $t\in \Omega_2$
	\begin{center}
	 $L-\epsilon<h(t)<L+\epsilon$.
	\end{center}
	But $\phi \notin F(I)$ implies $\Omega_1\cap\Omega_2\neq \phi$ hence for all $t\geq t^* =Max\left\{t_1^*,t_2^*\right\} $ and $t\in \Omega_1\cap\Omega_2\in \mathbb F(I)$ we have
	\begin{center}
	 $L-\epsilon<f(t)<L+\epsilon$
	\end{center}
	 and
	 \begin{center}
		$L-\epsilon<h(t)<L+\epsilon$ 
	 \end{center}
	then 
	\begin{center}
	$L-\epsilon<f(t)\leq g(t)\leq h(t)<L+\epsilon$
	\end{center}
	hence 
		\begin{center}
		$I^*-limg(t)=L$.
		\end{center}
Which completes the proof.		 
\end{proof}
\begin{corr}
 Let $I$ be a $B$-admissible ideal and let $f, g, h$ are all $\Delta$-measurable function on $T$. If $f(t)\leq g(t)\leq h(t)$ for all $t \in T$ and if 
\begin{center}
$I-lim f(t)=L=I-lim h(t) $
\end{center}
 then $I-lim g(t)=L$.
\end{corr}
\begin{defn}
  An $B$-admissible ideal $I$ is said to satisfy the condition $(BAP)$ if for every countable family of mutually disjoint sets 
	\begin{center}
	$\left\{A_1;A_2; : : : \right\}$
	\end{center}
	 belonging to $I$ there exists a countable family of sets 
	 \begin{center}
	 $\left\{B_1;B_2; : : : \right\}$
	 \end{center}
		from $T$ such that for each $j=1,2,..$ $A_j\Delta B_j$ are bounded and $\bigcup_{j=1}^\infty B_j \in I$
\end{defn}
Note here $A_j\Delta B_j$ means symmetric difference between them.
%.........................................theo3.5.................................
\begin{thm}
 If an $B$-admissible ideal satisfy the condition $(BAP)$ then $I$-convergence of a $\Delta$-measurable function $f$ implies $I^*$-convergence of $f$.
\end{thm}
\begin{proof}
Let $I-lim f(t)=L$. Then 
\begin{center}
$A(\epsilon)=\left\{t\in T :\left|f(t)-L\right|\geq\epsilon\right\}\in I$.
\end{center}
 Put 
 \begin{center}
 $A_1=\left\{t\in T:\left|f(t)-L\right|\geq 1 \right\}$
 \end{center}
	and 
	\begin{center}
	$A_n=\left\{t\in T :\frac{1}{n}\leq\left|f(t)-L\right|<\frac{1}{n-1}\right\} $ for $ n\geq 2$.
	\end{center}
	Obviously $A_i\cap A_j=\phi$ for $i\neq j$. By condition $(BAP)$ there exists a sequence of sets $ (B_n)_{n\in \mathbb N}$ such that $A_j\Delta B_j$ are bounded sets and $B=\bigcup_{i=1}^\infty B_j\in I$. It is sufficient to prove that for $t\in T\setminus B=M\in\mathbb F(I)$
		\begin{equation}
		\underset{t\rightarrow\infty}{\lim}lim f(t)=L.~~~~~~~~~~~~~~~~~~~~~~~~~~~~~~~~~~~~~~~~~~~~~~~~~~
	\end{equation}	
		 Let $\eta\geq 0$. Choose $k\in \mathbb N$ such that $\frac{1}{k+1}<\eta$. Then 
	\begin{center}
	$\left\{t\in T :\left|f(t)-L\right|\geq\eta\right\}\subset\cup_{j=1}^{k+1}A_j$.
	\end{center}
	Since $A_j\Delta B_j$ for $j=1,2,....k+1$ are bounded sets so  there exists $ t^*\in T$ such that
	\begin{equation}
	(\bigcup_{j=1}^{k+1}B_j)\cap\left\{t\in T : t>t_0^*\right\}=(\bigcup_{j=1}^{k+1}A_j)\cap\left\{t\in T : t>t_0^*\right\}~~~~~~~~~~~~~.
	\end{equation}
	If $t>t_0^*$ and $t\notin B$, then $t\notin \bigcup_{j=1}^{k+1}B_j$ and by (2) $t\notin \bigcup_{j=1}^{k+1}A_j$. But then $\left|f(t)-L\right|<\frac{1}{k+1}<\eta$. So $I^*-lim f(t)=L$. 
\end{proof}
\begin{defn}
 A $\Delta$-measurable function $f$ on time scales $T$ is said to be $I$-Cauchy if for every $\epsilon>0$, there exists a $t_1>t_0$
such that 
\begin{center}
$\left\{t:\left|f(t)-f(t_1)\right|\geq\epsilon \right\}\in I$.
\end{center}
 %.......................................theo3.6........................................
\end{defn}
\begin{thm} 
If $f:T\rightarrow \mathbb R$ a $\Delta$-measurable function then $I$-convergent of $f$ implies $I$-Cauchy.
\end{thm}
\begin{proof}
Let $\epsilon>0$ be given and also let $I-lim f(t)=L$. Then 
\begin{center}
$A=\left\{t\in T:\left|f(t)-L\right|\geq \frac{\epsilon}{2}\right\}\in I$.
\end{center}
Thus for all $t\notin A$ we have
\begin{center}
 $\left|f(t)-L\right|<\frac{\epsilon}{2}$.
\end{center}
 Choose  an element $t_1\in T\setminus A$ and $t_1>t_0$ implies 
 \begin{center}
 $\left|f(t_1)-L\right|<\frac{\epsilon}{2}$.
 \end{center}
Thus for all $t\notin A$ and $t_1>t_0$ we have 
\begin{center}
$\left|f(t)-f(t_1)\right|\leq\left|f(t)-L\right|+\left|f(t_1)-L\right|<\frac{\epsilon}{2}+\frac{\epsilon}{2}$.
\end{center}
Let $B=\left\{t\in T:\left|f(t)-f(t_1)\right|\geq\epsilon\right\}.$ Then clearly $B\subset A$ and since $A\in I$ we have $B\in I$. This completes the proof.
\end{proof}
In our final theorem we show that in some special condition a $I$-Cauchy $\Delta$-measurable function is $I$-convergent. Before that we first introduce a definition.
%.......................................defn......................
\begin{defn}
$L\in \mathbb R$ is called $I$-cluster point of a $\Delta$-measurable function $f$ if for every $\epsilon >0$ 
\begin{center}
$\left\{t:\left|f(t)-L\right|<\epsilon \right\}\notin I$.
\end{center}
\end{defn}
%..............................................them.....................3.7....................
\begin{thm}
If a $I$-Cauchy $\Delta$-measurable function has a $I$-cluster point then the function is $I$-convergent .
\end{thm}
\begin{proof}
Let a function $f$ be a $\Delta$-measurable function on a time scales $T$ and has a $I$-cluster point $L$ and be $I$-Cauchy. Let $\epsilon>0$ be given, we will show that
\begin{center}
 $\left\{t:\left|f(t)-L\right|\geq\epsilon\right\}\in I$.
\end{center}
Since $f$ is $I$-Cauchy  there exists $t_2\in T$ such that
\begin{center}
 $\left\{t:\left|f(t)-f(t_2)\right|\geq \frac{\epsilon}{4}\right\}\in I$ 
\end{center}
i.e. 
\begin{center}
$\left\{t:\left|f(t)-f(t_2)\right|<\frac{\epsilon}{4}\right\}\in \mathbb F(I)$.
\end{center}
Setting $A=\left\{t:\left|f(t)-f(t_2)\right|\geq\frac{\epsilon}{4}\right\}$ we have $A\in I$ and $s,t_1\notin A$ implies 
\begin{center}
$\left|f(s)-f(t_2)\right|<\frac{\epsilon}{4}$ and $\left|f(t_1)-f(t_2)\right|<\frac{\epsilon}{4}$ 
\end{center}
which implies $\left|f(s)-f(t_1)\right|<\frac{\epsilon}{2}$. Let $B=\left\{t:\left|f(t)-L\right|<\frac{\epsilon}{2}\right\}$. Since $L$ is a $I$-cluster point 
so it follows that $B\notin I$ and $A^c\in \mathbb F(I)$, then clearly $B\cap A^c\neq\phi$. Choose $s\in B\cap A^c$ then 
\begin{center}
$\left|f(s)-L\right|<\frac{\epsilon}{2}$ and $\left|f(s)-f(t_1)\right|<\frac{\epsilon}{2}$
\end{center}
for arbitrary $t_1\in A^c$. So $\left|f(t_1)-L\right|<\epsilon$ thus $t_1\in \left\{t:\left|f(t)-L\right|<\epsilon \right\}$ this shows that $A^c\subset \left\{t:\left|f(t)-L\right|<\epsilon \right\}$. Since $A^c\in \mathbb F(I)$,
\begin{center}
 $\left\{t:\left|f(t)-L\right|<\epsilon \right\}\in \mathbb F(I)$
\end{center}
 i.e
 \begin{center}
	$\left\{t:\left|f(t)-L\right|\geq\epsilon \right\}\in I$.
 \end{center}
 Therefore $f$ is $I$-convergent to $L.$
\end{proof}

\noindent\textbf{Acknowledgement:} The authors are grateful to
Dr. Prasanta Malik for his advice during the preparation of
this paper. The work of the first author was supported by University Grant Commission, New Delhi, India under UGC JRF in Mathematical Sciences. \\

%--------------------------------------------------REFERENCES------------------------------

\end{document}